\documentclass[11pt]{amsart}

\usepackage{amssymb}
\usepackage{amsmath}
\usepackage{amsthm}
\usepackage{enumerate}
\usepackage{graphicx}
\usepackage{caption}
\usepackage{subcaption}
 
\usepackage{hyperref}  
 
\hypersetup{colorlinks=true, citecolor=blue}

\theoremstyle{plain}
\newtheorem{thm}{Theorem}[section]
\newtheorem{prop}[thm]{Proposition}
\newtheorem{lem}[thm]{Lemma}

\numberwithin{equation}{section}

\theoremstyle{definition}

\theoremstyle{remark}

\newtheorem*{acknowledgements}{Acknowledgements}
\newtheorem*{funding}{Funding}

\theoremstyle{plain}

\newcommand{\thmref}[1]{Theorem~\ref{#1}}

\newcommand{\propref}[1]{Proposition~\ref{#1}}

\newcommand{\lemref}[1]{Lemma~\ref{#1}}

\newcommand{\calR}{{\mathcal R}}

\newcommand{\X}{{\mathcal X}}

\newcommand{\T}{{\mathcal T}}

\newcommand{\RR}{{\mathbb R}}

\newcommand{\eps}{{\varepsilon}}

\begin{document}

\title{The dimension of Thurston's spine}

\author{Maxime Fortier Bourque}
\address{D\'epartement de math\'ematiques et de statistique, Universit\'e de Montr\'eal, 2920, chemin de la Tour, Montr\'eal (QC), H3T 1J4, Canada}
\email{maxime.fortier.bourque@umontreal.ca}

\begin{abstract}
We show that for every $\eps>0$, there exists some $g\geq 2$ such that the set of closed hyperbolic surfaces of genus $g$ whose systoles fill has dimension at least $(5-\eps) g$. In particular, the dimension of this set---proposed as a spine for moduli space by Thurston---is larger than the virtual co\-homological dimension of the mapping class group.
\end{abstract}

\maketitle

\section{Introduction}

A \emph{systole} in a closed hyperbolic surface is a closed geodesic of minimal length. The word ``systole'' is also used for the common length of these geodesics. A set of closed geodesics \emph{fills} if each component of the complement of their union is contractible, in other words, if it cuts the surface into polygons. Let $\T_g$ be the Teichm\"uller space of closed hyperbolic surfaces of genus $g$ and $\X_g$ the subset of such surfaces whose systoles fill. In a 3-page preprint \cite{Thurston}, Thurston claimed to prove that $\T_g$ admits a mapping class group equivariant deformation retract into $\X_g$. The idea of the proof is to construct a continuous mapping class group invariant vector field on $\T_g$ which vanishes only on $\X_g$ and such that the systole length increases along flow lines. As pointed out in \cite{Ji}, this is not enough to guarantee the existence of a deformation retraction into $\X_g$ since one can construct a continuous vector field on a closed disk which vanishes only on the boundary, yet the disk does not deformation retract into its boundary. Another reason to be skeptical is that since the systole is a topological Morse function on $\T_g$ \cite{Akrout}, one can also define an invariant vector field that vanishes only at the critical points, but $\T_g$ does not deformation retract into this infinite discrete set. The problem in both cases is that some flow lines flow away from the zero set initially even though they eventually flow back into it elsewhere. In general, finding spines is a very delicate business \cite{SoutoPettet1,SoutoPettet2,Lacoste}.

The goal of this paper is to prove the following lower bound on the dimension of $\X_g$ in certain genera.

\begin{thm} \label{thm:dimension}
For every $\eps>0$, there exists an integer $g\geq 2$ such that $\X_g$ has dimension at least $(5-\eps)g$. 
\end{thm}

In particular, there exist infinitely many genera $g$ such that the dimension of $\X_g$ is strictly larger than $4g-5$, contrary to an earlier claim by Irmer in Version 1 of \cite{Irmer}. The number $4g-5$ is important because it is equal to the virtual cohomological dimension of the mapping class group, which is theoretically the minimum possible for an equivariant spine of $\T_g$.  We had previously shown that the dimension of $\X_g$ is at least $4g-5$ when $g$ is even and conjectured that its dimension is at least $(6-\eps)g$ in certain genera $g$ for every $\eps>0$ \cite[Theorem 7.1 and Conjecture 1.2]{sublinear}.

The proof of \thmref{thm:dimension} works by finding surfaces with roughly $g$ systoles that fill and then applying the implicit function theorem to get a submanifold of codimension roughly $g$ where these curves remain systoles. These surfaces are built using a simplified variant of a construction from \cite{sublinear}. The part that was missing in \cite{sublinear} to prove \thmref{thm:dimension} was a submersion property, which we were not able to prove because of the complexity of the graphs used to glue the surfaces. This is remedied here by simplifying the structure of the graphs. The drawback is that the number of systoles gets larger, so we do not obtain the full $(6-\eps)g$ conjecture.

\section{Surfaces from maps of large girth}  \label{sec:gluing}

In this section, we give a simple construction of a hyperbolic surface $X(t,M)$ that depends on a parameter $t>0$ and a graph $M$ satisfying certain conditions such that for some $t_0$ the systoles fill. This construction is a special case of a more general one alluded to in \cite[Remark 4.3]{sublinear}. We give the full details here for completeness.

For every integer $q\geq 3$ and parameter $t>0$, there exists a unique right-angled polygon $P(t,q)$ with $2q$ sides, half of which have length $t$ and the other half some number $s(t)$, alternately. To construct $P(t,q)$, start with a quadrilateral $Q(t,q)$ with one vertex $v_0$ of angle $\pi/q$, three right angles, and one of the sides opposite to $v_0$ of length $t/2$, then apply repeated reflections in the sides adjacent to $v_0$. To see why $Q(t,q)$ exists and is unique, we start with two rays $r_1$, $r_2$ at angle $\pi/q < \pi/2$ from the point $v_0$. There is a point $w$ on the ray $r_2$ such that the geodesic $\Gamma_w$  through $w$ and orthogonal to $r_2$ is asymptotic to $r_1$. As $w$ moves away from $v_0$ along $r_2$, the distance between $\Gamma_w$ and $r_1$ is strictly increasing, continuous, and varies from $0$ to $\infty$. By the intermediate value theorem, there exists a choice of $w$ such that this distance is exactly $t/2$ and this choice is unique by strict monotonicity.

We will need the fact that the common length $s(t)$ of the other sides of $P(t,q)$ is a strictly decreasing function of $t$. One reason why this is true is because the quadrilaterals $Q(t,q)$ all have the same area, so none of them is contained in any other. This means that as one side gets further from $r_1$, the other side gets closer to $r_2$. In fact, we have $\cos(\pi/q) = \sinh(t/2)\sinh(s(t)/2)$ \cite[Equation 2.3.1(i) on p.454]{Buser}. We will say that the sides of $P(t,q)$ of length $t$ are \emph{blue} and those of length $s(t)$ are \emph{red}. 

Let $p,q\geq 3$ be integers and let $M$ be a finite orientable surface map of type $\{p,q\}$ and girth $p$. This means that $M$ is a finite graph embedded in a surface such that each vertex has degree $q$, each complementary face has $p$ sides, and we require that every embedded cycle in $M$ has length at least $p$. The existence of such an $M$ is proved in \cite[Theorem 11]{Evans}. One can further require that $M$ is flag-transitive, but this will not be needed here.

We now describe a hyperbolic surface $B(t,M)$ with boundary obtained by gluing copies of $P(t,q)$ along $M$. To each vertex $v\in M$ corresponds a copy of $P_v$ of $P(t,q)$. There is a cyclic ordering of the edges adjacent to $v$ in $M$ coming from its embedding in a surface. We thus can thus associate the blue sides of $P_v$ with the edges adjacent to $v$ in a way that respects this cyclic ordering. For every edge $e=\{v,w\}$ in $M$, we glue the polygons $P_v$ and $P_w$ along their side associated to $e$ in a way that respects the orientations on $P_v$ and $P_w$. Topologically, the resulting surface $B(t,M)$ is a thickening of the map $M$. Geometrically, it is a compact hyperbolic surface with geodesic boundary tiled by copies of $P(t,q)$. Each boundary component of $B(t,M)$ is a concatenation of exactly $p$ red sides of polygons $P(t,q)$ as it corresponds to a face of $M$. Its length is thus $s(t) p$.

Finally, the surface $X(t,M)$ is defined as the double of $B(t,M)$ across its boundary. This means that we take two copies of $B(t,M)$ with opposite orientations and glue corresponding boundary components without twist. The blue arcs in the two copies of $B(t,M)$ line up in pairs to form closed geodesics of length $2t$ in $X(t,M)$.

Another way to think of $X(t,M)$ is to first double $P(t,q)$ across its red sides to form a sphere $S(t,q)$ with $q$ blue boundary components that come with a natural cyclic ordering. Then $X(t,M)$ is obtained by gluing copies of $S(t,q)$ along $M$. The special case where $q=3$ is the familiar way of assembling a surface out of pairs of pants without twists.

Let us count the number of red and blue curves in $X(t,M)$ in terms of its genus.
\begin{lem} \label{lem:number}
If $g$ is the genus of $X(t,M)$, then its number of blue curves is $\frac{q}{q-2}(g-1)$ and its number of red curves is $\frac{2q}{p(q-2)}(g-1)$.
\end{lem}
\begin{proof}
By the Gauss--Bonnet formula, the area of $X(t,M)$ is $4\pi(g-1)$ and the area of $P(t,q)$ is $\pi(q-2)$. Since there are $2V(M)$ polygons in $X(t,M)$ where $V(M)$ is the number of vertices in $M$, we have 
\[2(g-1)=(q-2)V(M).\]

The number of blue geodesics in $X(t,M)$ is $E(M)$, the number of edges in $M$, while the number of red geodesics is $F(M)$, the number of faces in $M$. Since each edge of $M$ belongs to two vertices and two faces, and each face has $p$ edges and each vertex belongs to $q$ edges, we have
\[
 pF(M) = 2E(M) = qV(M).
\]
We thus get
\[
E(M) = \frac{q}{2}V(M) = \frac{q}{q-2}(g-1) \text{ and } F(M) = \frac{q}{p} V(M) = \frac{2q}{p(q-2)}(g-1)
\]
as required.
\end{proof}

By taking $p$ and $q$ sufficiently large, we can make sure that there are fewer than $(1+\frac{\eps}{2})(g-1)$ blue curves and $\frac{\eps}{2}(g-1)$ red curves for any $\eps>0$.

\section{Systoles} 

We now determine which curves are the systoles in $X(M)=X(t_0,M)$ where $t_0$ is chosen in such a way that $2t_0 = s(t_0)p$, so that the red and blue curves have the same length.

The following lemma was proved in \cite[Lemma 2.3]{sublinear} in the symmetric case $t=s(t)$, but the same argument works in the asymmetric case. We repeat it for the sake of completeness.

\begin{lem}
Every arc in $P(t,q)$ between different red sides has length at least $t$ with equality only if it is a blue side and every arc between different blue sides has length at least $s(t)$ with equality only if it is a red side. 
\end{lem} 
\begin{proof}
Let $\alpha$ be a geodesic arc of minimal length between two red sides. Then $\alpha$ is orthogonal to the red sides of $P(t,q)$ at its endpoints. If $\alpha$ is not a blue side, then it joins two non-consecutive red sides. This implies that the rotation of angle $2\pi/q$ of $P(t,q)$ around its center sends $\alpha$ to an arc $\beta$ that intersects it, because their endpoints are intertwined along $\partial P(t,q)$. Following $\alpha$ initially and then continuing along $\beta$ after the point of intersection produces an arc $\gamma$ of the same length as $\alpha$ and $\beta$ with endpoints in distinct red sides. Since $\gamma$ is not geodesic, we can produce a strictly shorter arc, which is a contradiction. We conclude that $\alpha$ is a blue side. The other statement follows by exchanging the words ``red'' and ``blue'' in the above argument.
\end{proof}

We then move on to the double $S(t,q)$ of $P(t,q)$ across the red sides. The proof of the following lemma is identical to analogous results in \cite{sublinear}. 

\begin{lem} \label{lem:systoles_in_block}
For every $t>0$, the systoles in $S(t,q)$ are the boundary geodesics of length $2t$. The shortest arcs between distinct boundary components of $S(t,q)$ are the red arcs of length $s(t)$. Finally, every arc from a boundary component to itself in $S(t,q)$ that cannot be homotoped into the boundary has length strictly larger than $t$.
\end{lem}
\begin{proof}
Let $\gamma$ be a systole in $S(t,q)$. Then $\gamma$ crosses at least two red arcs since the complement of any $(q-1)$ red arcs in $S(t,q)$ is simply connected. Any segment of $\gamma$ between two consecutive intersection points with the red arcs stays in one of the two polygons $P(t,q)$ and joins two distinct red sides, hence has length at least $t$ by the previous lemma. Thus, $\gamma$ has length at least $2t$ with equality only if it is a blue boundary geodesic.

Let $\alpha$ be an arc of minimal between distinct boundary components of $S(t,q)$. Such an arc is necessarily geodesic and orthogonal to the boundary at endpoints. If $\alpha$ is not contained in one of the two copies of $P(t,q)$, then there is a non-geodesic arc of the same length that starts along $\alpha$ and continues along its reflection across the red arcs after an intersection point, which is a contradiction. We conclude that $\alpha$ is contained in one of the two copies of $P(t,q)$, so that $\alpha$ has length at least $s(t)$ by the previous lemma, with equality only if it is a red arc.

Finally, let $\alpha$ be a non-trivial arc from a boundary component $b$ to itself of length at most $t$. The curve formed by $\alpha$ together with the shorter of the two subarcs of $b$ between its endpoints is homotopically non-trivial and of length at most $2t$. By the first paragraph of this proof, this curve must be a boundary geodesic, which is a contradiction.  
\end{proof}

Let $X(M)=X(t_0,M)$ where $t_0$ is chosen as above so that $2t_0 = p\, s(t_0)$. Then the systoles in $X(M)$ are precisely the blue and the red curves.

\begin{prop} \label{prop:systoles}
Let $p,q\geq 3$ be integers and let $M$ be a map of type $\{p,q\}$ and girth $p$. Then the systoles in $X(M)$ are the blue curves and the red curves. In particular, they fill.
\end{prop}
\begin{proof}
Let $\gamma$ be a closed geodesic in $X(M)$ which is not a blue curve. Consider the combinatorial shadow $\sigma$ of $\gamma$ in the map $M$, a loop that keeps track of which subsurfaces $S(t_0,q)$ the curve $\gamma$ visits. This shadow is well-defined up to cyclic permutations since any eventual intersection between $\gamma$ and the blue curves is transverse. Thus if $\gamma$ intersects a blue curve, then it traverses from one copy of $S(t_0,q)$ to an adjacent one and the shadow $\sigma$ traces the corresponding edge in $M$. 

If $\sigma$ is non-contractible in $M$, then its length is at least $p$ by the hypothesis on the girth of $M$. This means that $\gamma$ contains at least $p$ segments that connect distinct boundary components in copies of $S(t_0,q)$. By \lemref{lem:systoles_in_block}, $\gamma$ has length at least $s(t_0) p$, with equality only if it is a red curve.

If $\sigma$ is contractible in $M$, then it lifts to the universal cover of $M$, which is a tree, hence $\sigma$ itself traces the contour of a finite tree. If this finite tree is a single point, then $\gamma$ is contained in a single copy of $S(t_0,q)$. However, by \lemref{lem:systoles_in_block}, every closed geodesic in the interior of $S(t_0,q)$ has length strictly larger than $2t_0$ (the length of a blue curve), so $\gamma$ is not a systole. Otherwise, $\sigma$ traces the contour of a non-trivial tree, hence it has at least two backtracking points where it traverses an edge back and forth (there is one backtrack for each leaf of the finite tree that is traced). These backtracking points correspond to two disjoint subarcs in $\gamma$ that go from one boundary component of a copy of $S(t_0,q)$ to itself. By \lemref{lem:systoles_in_block}, each of these two subarcs has length strictly larger than $t_0$, so $\gamma$ has length strictly larger than $2t_0 = s(t_0) p$ so it is not a systole.

We conclude that the shortest closed geodesics in $X(M)$ are the blue curves and the red curves. Since these curves cut the surface into polygons, they fill.
\end{proof}

\section{Deformations preserving the systoles}

Let $X=X(M)=X(M,t_0)$ be as before. We first perturb $X$ to a nearby surface $Y$ so that the systoles remain the same but they do not intersect at right angles anymore. 

\begin{lem} \label{lem:angle}
There exists a hyperbolic surface $Y$ whose systoles are in same homotopy classes as on $X$ and the angle of intersection between any two intersecting ones is a common angle $\theta < \frac{\pi}{2}$.
\end{lem}
\begin{proof}
Let $L_0=2t_0 = s(t_0) p$ be the systole length of $X$ and let $L_1 > L_0$ be the next shortest geodesic length. If $t>t_0$, then the blue curves in $X(M,t)$ have length $2t > L_0$ and the red curves have length $s(t)p < L_0$.

Perform a left twist deformation of length $r>0$ around each blue curve in $X(M,t)$. Let us temporarily call ``green'' the geodesics in the same homotopy classes as the red curves after twisting. We claim that all the green curves have the same length and furthermore their angle of intersection with blue curves is equal to a common value $\theta$. This is because each green curve follows a zigzag of red seams of length $s(t)$ and blue arcs of length $r$. The green curve is then obtained by connecting the midpoints of the red seams and blue segments, bounding a sequence of right triangles along the way (the resulting curve is geodesic and in the right homotopy class). The length $\mu$ of the green curves thus satisfies
\[
\cosh(\mu/(2p))=\cosh(r/2)\cosh(s(t)/2)
\]
by the hyperbolic Pythagoras formula. In particular, there exists a unique $r>0$ (depending continuously on $t$) such that $\mu = 2t$, the length of the blue curves. The counter-clockwise angle $\theta$ from green to blue can be calculated from the trigonometric formula 
\[\sinh(s(t)/2)= \sin(\theta) \sinh(\mu/(2p)).\]
It is obviously strictly smaller than $\pi/2$.

Now that we no longer have any use for the red seams, we will call the green curves ``red'' again. For every $t>t_0$, we have chosen $r(t)>0$ in such a way that the blue and red curves on the surface $Y_t$ obtained by twisting $X(M,t)$ by distance $r(t)$ around each blue curve have the same length $2t$. By Wolpert's lemma \cite[Lemma 3.1]{WolpertLengthSpectra}, the length of any closed geodesic changes by a factor at most $e^{2d(X,Y_t)}$ from $X$ to $Y_t$, where $d$ is the Teichm\" uller distance. Since this distance tends to zero as $t\to t_0$ and since the length of any closed geodesic other than the red and blue ones is at least $L_1$ on $X$, we can pick $t$ close enough to $t_0$ so that $2t < (L_0+L_1)/2$ and any other closed geodesic in $Y_t$ has length at least $(L_0+L_1)/2$. Let $Y = Y_{t}$ for any such $t$.
\end{proof}

We now prove that at the point $Y$, the map that records the lengths of the red curves restricted to the submanifold $W$ of the Teichm\"uller space $\T_g$ where the blue curves remain of length $2t$ is a submersion, as long as $q$ is odd.

\begin{prop} \label{prop:submersion}
Let $\calR$ be the set of red curves in $Y$. If $q$ is odd, then the map $L : W \to \RR^\calR$ defined by $L(Z)=(\ell_\alpha(Z))_{\alpha\in \calR}$ is a submersion at the point $Y$, where $\ell_\alpha(Z)$ is the length of the geodesic $\alpha$ on $Z$.
\end{prop}
\begin{proof}
We will show that for every red curve $\alpha \in \calR$, the basis vector $e_\alpha$ with a one in the entry $\alpha$ and zeros elsewhere is in the image of the differential $d_Y L$. To produce $e_\alpha$, we will use the left twist deformations $\tau_\beta$ about some blue curves $\beta$, which are tangent vectors to paths that remain in $W$. Recall that $\alpha$ corresponds to a face in the map $M$. Pick any vertex $v$ along this face. The $q$ edges adjacent to $v$ in $M$ correspond to $q$ blue curves in $Y$ that cut out a sphere with $q$ holes. Enumerate these blue curves as $\beta_1,\ldots,\beta_q$ following the cyclic order around $v$ such that $\beta_1$ and $\beta_q$ are the two curves intersecting $\alpha$. Note that adjacent faces in $M$ are distinct by the hypothesis that its girth is $p$. That is, no face is glued to itself. This is why the $q$ faces around $v$ are distinct. Let us label $\alpha_1$ to $\alpha_q$ the $q$ red curves corresponding to these $q$ faces around $v$, with $\alpha_j$ intersecting precisely $\beta_{j-1}$ and $\beta_j$, where indices are taken modulo $q$, so that $\alpha_1=\alpha$.

Consider the tangent vector $\xi = \sum_{j=1}^q (-1)^j \tau_{\beta_j}$. The cosine formula of Wolpert \cite{WolpertTwist} (see also \cite{Kerckhoff}) states that for any closed geodesic $\gamma$ transverse to a simple closed geodesic $\beta$, we have
\[
d\ell_\gamma(\tau_\beta) = \sum_{z \in \gamma\cap \beta} \cos \angle_ z(\gamma,\beta)
\]
where $\angle_ z(\gamma,\beta)$ is the counter-clockwise angle from $\gamma$ to $\beta$ at the point $z$ and the sum runs over all intersection points. Any of the blue curves $\beta_j$ intersects only the two red curves $\alpha_{j-1}$ and $\alpha_j$ (each one in a single point) and the angle from these red curves to $\beta_j$ at the intersection points is equal to some fixed $\theta \neq \frac{\pi}{2}$ by \lemref{lem:angle}.

For a red curve $\gamma$ different from $\alpha_1, \ldots, \alpha_q$, we have $d\ell_\gamma(\xi)=0$ as there is no intersection between $\gamma$ and the curves $\beta_j$. For $1<j\leq q$, we have 
\begin{align*}
d\ell_{\alpha_j}(\xi) &= (-1)^{j-1} d\ell_{\alpha_j}(\tau_{\beta_{j-1}}) +(-1)^j d\ell_{\alpha_j} (\tau_{\beta_{j}}) \\
&= (-1)^{j-1} \cos \theta +(-1)^j \cos \theta = 0,
\end{align*}
while for $j=1$ we get 
\[
d\ell_{\alpha}(\xi) = d\ell_{\alpha_1}(\xi) =  (-1)^q d\ell_{\alpha}(\tau_{\beta_q}) - d\ell_{\alpha}(\tau_{\beta_1}) = - 2\cos \theta \neq 0
\]
since $q$ is odd.

We have thus proved that a non-zero multiple of the basis vector $e_\alpha$ is in the image of $d_YL$ and hence that $d_YL$ is surjective since $\alpha \in \calR$ was arbitrary.
\end{proof}

\thmref{thm:dimension} follows easily from this proposition.

\begin{proof}[Proof of \thmref{thm:dimension}]
Let $Y$ be as above with $q$ odd. Note that $W$ is a smooth submanifold of dimension $(6- \frac{q}{q-2})(g-1)$ inside $\T_g$ since we can complete the set of blue curves into a pair of pants decomposition to define Fenchel--Nielsen coordinates, and $W$ is the submanifold obtained by fixing $\frac{q}{q-2}(g-1)$ of the coordinates constant equal to $2t$, namely, the lengths of the blue curves (the number of which was determined in \lemref{lem:number}).

 By \propref{prop:submersion}, $L: W \to \RR^{\calR}$ is a submersion at the point $Y$, where it takes the value $L(Y)=(2t,\ldots,2t)$. By the implicit function theorem, near $Y$ the inverse image $L^{-1}(2t,\ldots,2t)$ is a smooth submanifold of $W$ of codimension $|\calR| = \frac{2q}{p(q-2)}(g-1)$ (see \lemref{lem:number}). Thus, there is a cell $C$ of dimension
\[ 
\left(6- \frac{q}{q-2}- \frac{2q}{p(q-2)}\right)(g-1)
\] 
through $Y$ in Teichm\"uller space where the red and blue curves all remain of length $2t$. For any $\eps>0$, we can choose $q$ odd large enough and $p$ large enough so that the above number is at least $(5-\eps)g$.

 By Wolpert's lemma, close enough to $Y$ in $C$, the systoles will remain the red and blue curves and thus fill. This shows that $\X_g$ contains a cell of dimension at least $(5-\eps)g$.
\end{proof}

\begin{funding}
The author was partially supported by Discovery Grant RGPIN-2022-03649 from the Natural Sciences and Engineering Research Council of Canada.
\end{funding}

\begin{acknowledgements}
I thank Ingrid Irmer for posting the preprint \cite{Irmer}, which prompted me to think about this problem again.
\end{acknowledgements}

\bibliographystyle{amsalpha}
\bibliography{biblio}

\end{document}